\documentclass[12pt,a4paper]{article}
\usepackage[utf8]{inputenc}
\usepackage{amsmath, amssymb, amsthm, verbatim, hyperref}
\usepackage{mathrsfs}
\usepackage[dvipsnames]{xcolor}
\usepackage{ragged2e}
\usepackage{marginnote}
\usepackage{enumerate}
\usepackage{makecell}
\usepackage{longtable}
\usepackage{epsfig}
\usepackage{tikz}
\usepackage{ytableau}
\usepackage{hyperref}
\usepackage{caption}
\usepackage{subcaption}
\hypersetup{hidelinks}
\usepackage{fancyvrb}
\usepackage{multirow}
\usepackage{tikz-cd}
\usepackage[left=2.50cm, right=2.50cm, top=2.00cm, bottom=2.00cm]{geometry}
\usetikzlibrary {arrows.meta}
\usepackage{ifthen}

\newcommand{\cross}[3]{
    \draw[thick] (-0.5+#1, -#2) to (0.5+#1, -#2+1);
    \draw[thick] (-0.5+#1, -#2+1) to (0.5+#1, -#2);
    \foreach \j in {0,...,#3}{
    \ifthenelse{\equal{\j}{\the\numexpr #2-1} \OR \j = #2}{}{
    \draw[thick] (-0.5+#1,-\j ) to (0.5+#1,-\j);}
    }
}

\newcommand{\uncross}[3]{
    \draw[thick] [bend right = 90, looseness=1.25] (-0.5+#1, -#2) to (-0.5+#1, -#2+1);
    \draw[thick] [bend right = 90, looseness=1.25] (0.5+#1, -#2+1) to (0.5+#1, -#2);
    \foreach \j in {0,...,#3}{
    \ifthenelse{\equal{\j}{\the\numexpr #2-1} \OR \j = #2}{}{
    \draw[thick] (-0.5+#1,-\j ) to (0.5+#1,-\j);}
    }
}

\definecolor{goodgreen}{rgb}{0.01, 0.75, 0.24}

\theoremstyle{plain}

\newtheorem{thm}{Theorem}[section]
\newtheorem*{thm*}{Theorem}
\newtheorem{prop}[thm]{Proposition}

\newtheorem{lemma}[thm]{Lemma}
\newtheorem{conjecture}[thm]{Conjecture}

\theoremstyle{definition}
\newtheorem{definition}[thm]{Definition}
\newtheorem{example}[thm]{Example}
\theoremstyle{remark}

\newtheorem{remark}[thm]{Remark}

\newcommand{\TTT}{\mathcal{T}}
\newcommand{\JJJ}{\mathcal{J}}
\newcommand{\MMM}{\mathcal{M}}

\newcommand{\xx}{\mathbf{x}}
\newcommand{\mm}{\mathbf{m}}

\title{\vspace{-1em} 
Correlations in random cluster model at $q=1$}
\author{\vspace{-2em} \\
Son Nguyen, Pavlo Pylyavskyy}
\date{\vspace{-3em}}

\allowdisplaybreaks

\begin{document}
\ytableausetup{centertableaux}

\maketitle

    \begin{abstract}
 Let $\mu$ be a measure that samples a subset of a finite ground set, and let $\mathcal{A}_e$ be the event that element $e$ is sampled. The measure $\mu$ is negatively correlated if for any pair of elements $e, f$ one has $\mu(\mathcal{A}_e \cap \mathcal{A}_f) - \mu(\mathcal{A}_e) \mu(\mathcal{A}_f) \leq 0$. A measure is positively correlated if the direction of the inequality is reversed.  
For the random cluster model on graphs positive correlation between edges is known for $q \geq 1$ due to the FKG inequality, while the negative correlation is only conjectured for $0 \leq q \leq 1$. The main result of this paper is to give a combinatorial formula for the difference in question at $q=1$. Previously, such a formula was known in the uniform spanning tree case, which is a limit of the random cluster model at $q=0$.
\end{abstract}

\tableofcontents

%--------------------------------------------------

\section{Introduction}\label{sec:intro}

An \emph{increasing event} \( \mathcal{A} \subseteq 2^{[n]} \)  
is a collection of subsets of \( [n] \) that are closed upward under containment, that is,
if \( A \in \mathcal{A} \) and \( A \subseteq B \subseteq [n] \), then \( B \in \mathcal{A} \).
Such an event depends only on some subset 
of the ground set,  
specifically on the union of the minimal elements of \( \mathcal{A} \) with respect to inclusion. A measure \( \mu \) on \( 2^{[n]} \) is called \emph{negatively associated} if it satisfies
\[
\mu(\mathcal{A} \cap \mathcal{B}) - \mu(\mathcal{A}) \mu(\mathcal{B}) \leq 0
\]
for any pair of increasing events \( \mathcal{A}, \mathcal{B} \) on \( 2^{[n]} \) that depend on disjoint sets of elements of the ground set. Similarly, $\mu$ is {\it positively associated} if it satisfies
\[
\mu(\mathcal{A} \cap \mathcal{B}) - \mu(\mathcal{A}) \mu(\mathcal{B}) \geq 0
\]
for any pair of increasing events \( \mathcal{A}, \mathcal{B} \) on \( 2^{[n]} \), without restrictions on the ground set.

The simplest example of an increasing event is the event $\mathcal{A}_e$ consisting of all subsets of $[n]$ containing specific element $e$. When a measure satisfies correct inequality 
for any elements $e \not = f$ of the ground set, we say that $\mu$ is {\it negatively correlated} or {\it positively correlated} respectively. 

The study of positive and negative association is a deep and beautiful subject. The case of positive association is easier due to a powerful of {\it {FKG inequality}}, due to Fortuin, Kastelyn, and Ginibre \cite{fortuin1971}. Specifically, if a measure satisfies certain local {\it positive lattice condition}, positive association follows. There are many other related interesting questions and results in the theory of positive association, such as Berg-Kesten inequality \cite{van1985}, Reimer inequality \cite{reimer2000}, etc. We direct the reader to \cite[Chapter 4]{grimmett2018} for a great introduction.

The case of negative association is harder, as no tool like FKG inequality is available. We direct the reader to the survey by Pemantle \cite{pemantle2000} and to the work of Borcea, Br\"and\'en, and Liggett \cite{borcea2009} for rich sources of information on major ideas, theorems, and conjectures in it. Among the notable  developments are the works of Wagner \cite{wagner2008} and Kahn and Neiman \cite{kahn2010}, as well as recent papers of Br\"and\'en and Huh \cite{branden2020} and of Huh, Schr\"oter, and Wang \cite{huh2021}. The first of the latter introduces Lorenzian polynomials as a powerful tool allowing to prove versions of negative association. 

The random cluster model, also known as Fortuin--Kasteleyn model \cite{fortuin1972}, is one of the fundamental models of statistical mechanics, see \cite{grimmett2006} for a definitive source.  
Negative association in the random cluster model was conjectured by Kahn \cite{kahn2000} and Grimmett--Winkler \cite{grimmett2004}. It is related to {\it uniform spanning tree} measure via certain limit at $q=0$. The uniform spanning tree measure is perhaps the most famous example of negative association, related to historical {\it Rayleigh monotonicity} property of electrical networks, see \cite{feder1992, choe2004, wagner2008, cibulka2008}, as well as \cite[Chapter 4]{lyons2017} for a modern exposition. Furthermore, one can also get {\it uniform spanning forest} and {\it uniform connected subgraph} measures by specializing random cluster model, for both of those negative association is conjectured and remains open \cite{harris1974, smirnov2006, lawler2010}. 

In the uniform spanning tree case an actual explicit formula is known \cite{feder1992, cibulka2008}. It can be formulated as follows, see the discussion preceding Theorem \ref{thm:trees} for the relevant notation. 

\begin{thm*}[\cite{feder1992}]
In the case of uniform spanning tree measure, for any graph $G$ and edges $e,f$ we have 
$$\mu(\mathcal{A} \cap \mathcal{B}) - \mu(\mathcal{A}) \mu(\mathcal{B})  \sim \left(\sum_{F \in \mathfrak F} \xx^F - \sum_{F \in \mathfrak F'} \xx^F \right)^2,$$ where $\mathfrak F$ and $\mathfrak F'$ are collections of forests satisfying certain properties with respect to the edges $e$ and $f$.
\end{thm*}

The main result of this paper is to give an explicit formula for the case $q=1$ as follows, see Section \ref{sec:formula} for relevant definitions and notation.

    \begin{thm*} [Theorem \ref{thm:main-thm}]
        For random cluster model at $q=1$ for any graph $G$ and edges $e,f$ we have
        \begin{equation}
           \mu(\mathcal{A} \cap \mathcal{B}) - \mu(\mathcal{A}) \mu(\mathcal{B})  \sim \sum_{\beta,\gamma}\left(\xx^{\beta}\xx^{\gamma}\sum_{\mm\in B_{\beta,\gamma}} \mm\right). 
        \end{equation}
    \end{thm*}

It is our hope that finding a common generalization of the two formulas can lead to resolving Kahn--Grimmett--Winkler conjecture in full generality. To this end we formulate $\alpha \beta \gamma$-ansatz conjecture, resolving which would generalize both of the above formulas, and would imply Kahn--Grimmett–-Winkler conjecture.

The paper is organized as follows. In Section \ref{sec:prelim} we explain the necessary background on random cluster model and correlations in it. In Section \ref{sec:formula} we state our formula for $q=1$ and in Section \ref{sec:proof} we prove it. In Section \ref{sec:conj} we state $\alpha \beta \gamma$-ansatz conjecture.

\section{Preliminaries} \label{sec:prelim}

\subsection{Background}

    Let $G = (V,E)$ be a graph and $\mathbf{x} = \{x_e:e\in E\}$ be a set of positive real numbers indexed by $E$. Here we allow $G$ to have parallel edges and loops. Let the state space be $\Omega = \{0,1\}^E$ consisting of vectors $\omega = (\omega(e):e\in E)$. Let $\eta(\omega) = \{e\in E:\omega(e) = 1\}$. This yields a correspondence between $\Omega$ and subsets $F\subseteq E$ given by $F = \eta(\omega)$. We will sometimes abuse the notation and denote $\eta(F)$ for the subgraph of $G$ induced by $F$. Let $F = \eta(\omega)$, we denote $\xx^{\omega}:=\xx^F := \prod_{e\in F}x_e$, and let $k(\omega):=k(F)$ be the number of connected components in $\eta(F)$.

    Given two subsets $A,B\subseteq E$, consider event
    \[ \JJJ_A^B = \{\omega\in\Omega:\omega(e) = 1~\text{for all $e\in A$}, \omega(e) = 0~\text{for all $e\in B$}\} \]
    and define the \textit{random cluster model} measure
    \[ \mu(\JJJ_A^B) = \TTT_A^B = \frac{1}{Z} \sum_{\omega\in\JJJ_A^B}\xx^{\omega}q^{k(\omega)}, \]
    where $Z = \sum_{\omega}\xx^{\omega}q^{k(\omega)}$ is the partition function. 
    \begin{remark}
Note that the measure for the random cluster model is usually written as 
$$\phi_{p,q}(\omega) = \frac{1}{Z} q^{k(\omega)} \prod_{e \in E} p^{\omega(e)} (1-p)^{1 - \omega(e)}.$$
One obtains this version from our version by setting $x_e = p/(1-p)$ for each edge $e$, and canceling the common factor $\prod_{e \in E} (1-p)$. 
    \end{remark}

    Fixing two edges $e,f\in E$, we abuse the notation and write $ef,e,f$ for $\{e,f\},\{e\},\{f\}$. For instance, we write $\JJJ_{ef},\JJJ^{ef},\JJJ_e^f,\JJJ_f^e$ for $\JJJ_{\{e,f\}}$, $\JJJ^{\{e,f\}}$, $\JJJ_{\{e\}}^{\{f\}}$, $\JJJ_{\{f\}}^{\{e\}}$, respectively. We say a probability measure $\mu$ is \textit{edge-negatively-associated} or {\it negatively correlated} if
    \[ \mu(\JJJ_e\cap\JJJ_f) \leq \mu(\JJJ_e)\mu(\JJJ_f),\quad\quad\text{for}~e,f\in E, e\neq f. \]

    \begin{conjecture}[\cite{kahn2000, grimmett2004}]\label{conj:old}
        The random cluster model is negatively correlated for $0 \leq q \leq 1$.
    \end{conjecture}

    With a little algebra, one can see that the condition for edge-negatively-association of random cluster model is equivalent to
    \[ \TTT_e^f\TTT_f^e - \TTT_{ef}\TTT^{ef} \geq 0. \]

    \begin{example} \label{ex:K3}
        Let $G = K_3$ with three edges $e,f,g$ as shown in Figure \ref{fig:K3}. Then
        \begin{align*}
            \TTT_e^f &= x_eq^2 + x_ex_gq\\
            \TTT_f^e &= x_fq^2 + x_fx_gq\\
            \TTT_{ef} &= x_ex_fq + x_ex_fx_gq\\
            \TTT^{ef} &= q^3 + x_gq^2
        \end{align*}
        Hence,
        \[ \TTT_e^f\TTT_f^e - \TTT_{ef}\TTT^{ef} = (x_eq^2 + x_ex_gq)(x_fq^2 + x_fx_gq) - (x_ex_fq + x_ex_fx_gq)(q^3 + x_gq^2) \]
        \[ = x_ex_fx_g(q^3-q^4) + x_ex_fx_g^2(q^2 - q^3), \]
        which is nonnegative for $0 \leq q \leq 1$.

        \begin{figure}[h!]
            \centering
            \includegraphics[scale = 0.8]{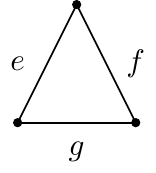}
            \caption{$K_3$}
            \label{fig:K3}
        \end{figure}
    \end{example}

\begin{remark} \label{remark:AB}
    Let $E^{ef}$ denote $E\backslash\{e,f\}$. One can easily see that there is a one-to-one correspondence between monomials in $\TTT_e^f\TTT_f^e$ and pairs of subset $(A,B)$ of $E^{ef}$ given by
    \[ (A,B)\mapsto q^{k_1(A,B)}\xx^{A+e}\xx^{B+f} \quad \text{where} \quad k_1(A,B) = k(A+e) + k(B+f). \] 
    We also have a similar correspondence with monomials in $\TTT_{ef}\TTT^{ef}$ given by
    \[ (A,B)\mapsto q^{k_2(A,B)}\xx^{A+e+f}\xx^{B} \quad \text{where} \quad k_2(A,B) = k(A+e+f) + k(B). \]
    Thus, when setting $q = 1$, every monomial appears in $\TTT_e^f\TTT_f^e$ and $\TTT_{ef}\TTT^{ef}$ with the same coefficient, and the difference is $0$. 
    \end{remark}
    
    It is easily seen in the Example \ref{ex:K3} above for example, as after setting $q=1$ the expression $x_ex_fx_g(q^3-q^4) + x_ex_fx_g^2(q^2 - q^3)$ turns into $0$. 

    By Factor theorem this means that $\TTT_e^f\TTT_f^e - \TTT_{ef}\TTT^{ef}$ is divisible by $1-q$. In addition, every monomial in this expression has a factor of $x_ex_f$. Thus, the difference $\TTT_e^f\TTT_f^e - \TTT_{ef}\TTT^{ef}$ is divisible by $x_ex_f(1-q)$, and thus for $0 < q < 1$ it is natural to reformulate the conjecture as positivity of the expression  
    \[ \MMM_{ef}(q) := (\TTT_e^f\TTT_f^e - \TTT_{ef}\TTT^{ef})/x_ex_f(1-q) \geq 0. \]
    It is useful to observe that each pair $(A,B)$ of subsets in $E^{ef}$ contributes $(q^{k_1(A,B)} - q^{k_2(A,B)})\xx^{A+B}/(1-q)$ to $\MMM_{ef}(q)$.

    \begin{example}\label{ex:k3q}
        For $G = K_3$, we have
        \[ \MMM_{ef}(q) := (\TTT_e^f\TTT_f^e - \TTT_{ef}\TTT^{ef})/(1-q) = x_gq^3 + x_g^2q^2. \]
    \end{example}

    \begin{conjecture} \label{conj:main}
        For $0 \leq q \leq 1$ and for positive edgeweights $x_e$ one has $\MMM_{ef}(q) \geq 0$. 
    \end{conjecture}

    Note that Conjecture \ref{conj:old} is trivial for $q=1$ while Conjecture \ref{conj:main} is not. In fact, the validity of Conjecture \ref{conj:main} follows, by taking $q \rightarrow 1$ limit, what is known about positive correlations for $q > 1$. 

\subsection{Comparison with Rayleigh monotonicity}

It is possible to take the limit in the random cluster model to obtain Uniform Spanning Tree (UST) model. Specifically, one needs to take the limit so that $p,q \rightarrow 0$ and $q/p \rightarrow 0$, see \cite[Theorem 1.23]{grimmett2006}. 

Negative correlation for the UST model is a classical result equivalent to Rayleigh monotonicity in electrical networks, we refer the reader to \cite{feder1992, choe2004, cibulka2008, wagner2008, borcea2009, lyons2017} for the proofs and the background. In our language, this result can be formulated as follows, where the limit and the specialization pick out the terms with the lowest power of $q$. 

\begin{thm}[{\cite[Theorem 2.1]{feder1992}}] \label{thm:trees}
We have $$(\TTT_e^f\TTT_f^e - \TTT_{ef}\TTT^{ef})/q^2 \mid_{q=0} \geq 0 \Leftrightarrow \frac{\MMM_{ef}}{q^2} \mid_{q=0} \geq 0.$$ In fact, $$\frac{\TTT_e^f\TTT_f^e - \TTT_{ef}\TTT^{ef}}{q^2 x_e x_f} \mid_{q=0}  = \left(\sum_{F \in \mathfrak F} \xx^F - \sum_{F \in \mathfrak F'} \xx^F \right)^2,$$ where $\mathfrak F$ and $\mathfrak F'$ are collections of forests satisfying certain properties with respect to the edges $e$ and $f$.
\end{thm}

In other words, in this case the negative correlation holds because the difference in question is a single square.

\begin{example}
In Example \ref{ex:K3} we have 
$$\frac{x_ex_fx_g(q^3-q^4) + x_ex_fx_g^2(q^2 - q^3)}{q^2x_e x_f} \mid_{q=0} = x_g^2.$$
This example is sufficiently small that there is no subtraction inside the square on the right. However, already for the complete graph on four vertices the expression squared is not subtraction-free. 
\end{example}

\section{Formula for $\MMM_{ef}(1)$} \label{sec:formula}

    \begin{definition}
        Fix two edges $e,f\in E$, a subset $F\subseteq E^{ef}$ is a \textit{paracel}\footnote{The word ``paracel'' is inspired by Vietnam's Paracel Islands, consisting of many small islands similar to connected components of a graph.} if the same two connected components of $\eta(F)$ are connected by both $e$ and $f$. Equivalently, $\eta(F+e)$ and $\eta(F+f)$ have the same connected components (as sets of vertices), and $$k(F+e) = k(F+f) = k(F) - 1.$$ Additionally, we call an edge in $E^{ef} - F$ a \textit{smoot} for a paracel $F$ if it also connects the same two connected components of $\eta(F)$ connected by $e$ and $f$.
    \end{definition}

    \begin{example}
        In Figure \ref{fig:paracel}, $e$ and $f$ are the red edges. The first two subgraphs from the left are paracels. The third subgraph is not a paracel since $e$ does not connect two connected components. The last subgraph is not a paracel since $e$ and $f$ do not connect the same two connected components.
    
        \begin{figure}[h!]
            \centering
            \includegraphics[scale = 0.8]{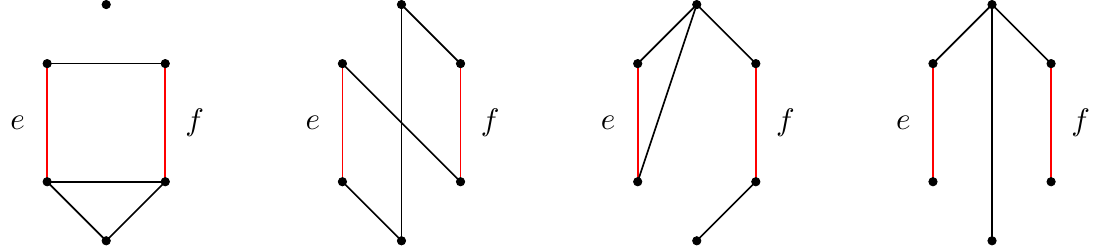}
            \caption{Examples and non-examples of paracels}
            \label{fig:paracel}
        \end{figure}
    \end{example}

    \begin{example}
        In Figure \ref{fig:smoots}, $e$ and $f$ are the red edges. The subgraph formed by the black edges is a paracel. The orange edges are smoots.
    
        \begin{figure}[h!]
            \centering
            \includegraphics[scale = 0.8]{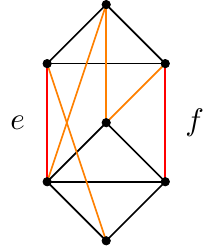}
            \caption{Smoots}
            \label{fig:smoots}
        \end{figure}
    \end{example}

    \begin{definition}\label{def:compatible}
        Given disjoint subsets $\beta,\gamma\subseteq E^{ef}$, a subset $\alpha\subseteq E^{ef}$ is \textit{compatible} with $\beta,\gamma$  if
        \begin{itemize}
            \item $\alpha$ is disjoint from both $\beta$ and $\gamma$;
            \item $\gamma + \alpha$ is a paracel;
            \item every edge in $\beta$ is a smoot in $\eta(\gamma + \alpha)$.
        \end{itemize}
        Let $A_{\beta,\gamma} = \{\alpha:\text{$\alpha$ is compatible with $\beta,\gamma$}\}$. Two subsets $\alpha,\alpha'\in A_{\beta,\gamma}$ are \textit{twins} if $\alpha\cap\alpha'$ is also in $A_{\beta,\gamma}$. Note that every $\alpha\in A_{\beta,\gamma}$ is a twin with itself since $\alpha\cap \alpha = \alpha$. Let $B_{\beta,\gamma}$ be the set of monomials that can be written as $\mathbf{x}^\alpha \mathbf{x}^{\alpha'}$ for some twins $\alpha,\alpha'\in A_{\beta,\gamma}$.
    \end{definition}

    The following theorem is the main theorem of this paper. Note that $\MMM_{ef}(1)$ is not only positive as a function, but in fact is a subtraction-free polynomial. This is not the case for $q=0$ as discussed above: while the lowest degree terms in $\MMM_{ef}(0)$ form a single square as in Theorem \ref{thm:trees}, the squared terms may come with opposite signs. 

    \begin{thm}\label{thm:main-thm}
        For any graph $G$ and edges $e,f$, we have
        \begin{equation}\label{eqn:main-eqn}
           \MMM_{ef}(1) = \sum_{\beta,\gamma}\left(\xx^{\beta}\xx^{\gamma}\sum_{\mm\in B_{\beta,\gamma}} \mm\right). 
        \end{equation}
    \end{thm}

    \begin{example}
        If $e$ or $f$ is a loop in $G$, then from Remark \ref{remark:AB}, the LHS is trivially $0$. For the RHS, there is no possible $\gamma$ and $\alpha$ such that $\gamma + \alpha$ is a paracel. Thus, the RHS is also $0$.
    \end{example}

    \begin{example}
        For $G = K_3$ with three edges $e,f,g$. The pairs of $\beta$ and $\gamma$ with nonempty $B_{\beta,\gamma}$ are
        \begin{center}
            \begin{tabular}{c | c | c | c} 
                 $\beta$ & $\gamma$ & $A_{\beta,\gamma}$ & $B_{\beta,\gamma}$ \\
                 \hline\hline
                 $\varnothing$ & $\varnothing$ & $\{\{g\}\}$ & $\{x_g^2\}$ \\
                 \hline
                 $\varnothing$ & $\{g\}$ & $\{\varnothing\}$ & $\{1\}$
            \end{tabular}
        \end{center}

        Hence,
        \[ \MMM_{ef}(1) = \xx^{\varnothing}\cdot \xx^{\varnothing}\cdot x_g^2 + \xx^{\varnothing}\cdot \xx^{\{g\}}\cdot 1 = x_g^2 + x_g, \]
        which is consistent with Example \ref{ex:k3q}.
    \end{example}
    
    \begin{example}\label{ex:K4-1}
        Let $G$ and $e,f$ as in Figure \ref{fig:K4minusOne}.
        
        \begin{figure}[h!]
            \centering
            \includegraphics[scale = 0.8]{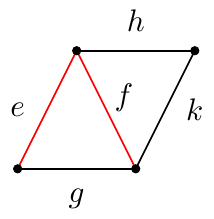}
            \caption{}
            \label{fig:K4minusOne}
        \end{figure}
        
        The pairs of $\beta$ and $\gamma$ with nonempty $B_{\beta,\gamma}$ are
        \begin{center}
            \begin{tabular}{c | c | c | c} 
                 $\beta$ & $\gamma$ & $A_{\beta,\gamma}$ & $B_{\beta,\gamma}$ \\
                 \hline\hline
                 $\varnothing$ & $\varnothing$ & $\{\{g\},\{g,h\},\{g,k\}\}$ & $\{x_g^2,x_g^2x_h^2,x_g^2x_k^2,x_g^2x_h,x_g^2x_k,x_g^2x_hx_k\}$ \\
                 \hline
                 $\varnothing$ & $\{g\}$ & $\{\varnothing,\{h\},\{k\}\}$ & $\{1,x_h^2,x_k^2,x_h,x_k,x_hx_k\}$ \\
                 \hline
                 $\varnothing$ & $\{h\}$ & $\{\{g\}\}$ & $\{x_g^2\}$ \\
                 \hline
                 $\varnothing$ & $\{k\}$ & $\{\{g\}\}$ & $\{x_g^2\}$ \\
                 \hline
                 $\varnothing$ & $\{g,h\}$ & $\{\varnothing\}$ & $\{1\}$ \\
                 \hline
                 $\varnothing$ & $\{g,k\}$ & $\{\varnothing\}$ & $\{1\}$ \\
                 \hline
                 $\{h\}$ & $\varnothing$ & $\{\{g,k\}\}$ & $\{x_g^2x_k^2\}$ \\
                 \hline
                 $\{h\}$ & $\{g\}$ & $\{\{k\}\}$ & $\{x_k^2\}$ \\
                 \hline
                 $\{h\}$ & $\{k\}$ & $\{\{g\}\}$ & $\{x_g^2\}$ \\
                 \hline
                 $\{h\}$ & $\{g,k\}$ & $\{\varnothing\}$ & $\{1\}$ \\
                 \hline
                 $\{k\}$ & $\varnothing$ & $\{\{g,h\}\}$ & $\{x_g^2x_h^2\}$ \\
                 \hline
                 $\{k\}$ & $\{g\}$ & $\{\{h\}\}$ & $\{x_h^2\}$ \\
                 \hline
                 $\{k\}$ & $\{h\}$ & $\{\{g\}\}$ & $\{x_g^2\}$ \\
                 \hline
                 $\{k\}$ & $\{g,h\}$ & $\{\varnothing\}$ & $\{1\}$ \\
            \end{tabular}
        \end{center}

        Hence,
        \[ \MMM_{ef}(1) = (x_g^2+x_g^2x_h^2+x_g^2x_k^2+x_g^2x_h+x_g^2x_k+x_g^2x_hx_k) \]
        \[ + x_g(1+x_h^2+x_k^2+x_h+x_k+x_hx_k) + x_hx_g^2 + x_kx_g^2 + x_gx_h + x_gx_k \]
        \[ + x_hx_g^2x_k^2 + x_hx_gx_k^2 + x_hx_kx_g^2 + x_hx_gx_k + x_kx_g^2x_h^2 + x_kx_gx_h^2 + x_kx_hx_g^2 + x_kx_gx_h. \]
    \end{example}

    \begin{example}
        In Figure \ref{fig:formulaEx1}, consider the pair $\beta = \{i\}$ and $\gamma = \{a,c,d\}$. The set of compatible $\alpha$ is
        \[ A_{\beta,\gamma} = \{\{b\},\{g\},\{b,g\},\{b,h\},\{g,h\},\{b,g,h\}\}. \]
        
        \begin{figure}[h!]
            \centering
            \includegraphics[scale = 0.8]{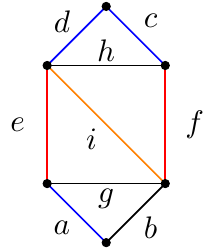}
            \caption{$\beta = \{i\}$ and $\gamma = \{a,c,d\}$}
            \label{fig:formulaEx1}
        \end{figure}
        
        The pairs of twins and the respective monomials are
        \begin{center}
            \begin{tabular}{c | c} 
                 Twins & Monomials \\
                 \hline\hline
                 $\{b\} , \{b\}$ & $x_b^2$ \\
                 \hline
                 $\{b\} , \{b,g\}$ & $x_b^2x_g$ \\
                 \hline
                 $\{b\} , \{b,h\}$ & $x_b^2x_h$ \\
                 \hline
                 $\{b\} , \{b,g,h\}$ & $x_b^2x_gx_h$ \\
                 \hline
                 $\{g\} , \{g\}$ & $x_g^2$ \\
                 \hline
                 $\{g\} , \{b,g\}$ & $x_bx_g^2$ \\
                 \hline
                 $\{g\} , \{g,h\}$ & $x_g^2x_h$ \\
                 \hline
                 $\{g\} , \{b,g,h\}$ & $x_bx_g^2x_h$ \\
                 \hline
                 $\{b,g\} , \{b,g\}$ & $x_b^2x_g^2$ \\
                 \hline
                 $\{b,g\} , \{b,h\}$ & $x_b^2x_gx_h$ \\
                 \hline
                 $\{b,g\} , \{g,h\}$ & $x_bx_g^2x_h$ \\
                 \hline
                 $\{b,g\} , \{b,g,h\}$ & $x_b^2x_g^2x_h$ \\
                 \hline
                 $\{b,h\} , \{b,h\}$ & $x_b^2x_h^2$ \\
                 \hline
                 $\{b,h\} , \{b,g,h\}$ & $x_b^2x_gx_h^2$ \\
                 \hline
                 $\{g,h\} , \{g,h\}$ & $x_g^2x_h^2$ \\
                 \hline
                 $\{g,h\} , \{b,g,h\}$ & $x_bx_g^2x_h^2$ \\
                 \hline
                 $\{b,g,h\} , \{b,g,h\}$ & $x_b^2x_g^2x_h^2$
            \end{tabular}
        \end{center}

        Thus, $B_{\beta,\gamma} = \{x_b^2,$ $x_b^2x_g,$ $x_b^2x_h,$ $x_b^2x_gx_h,$ $x_g^2,$ $x_bx_g^2,$ $x_g^2x_h,$ $x_bx_g^2x_h,$ $x_b^2x_g^2,$ $x_b^2x_g^2x_h,$ $x_b^2x_h^2,$ $x_b^2x_gx_h^2,$ $x_g^2x_h^2,$ $x_bx_g^2x_h^2,$ $x_b^2x_g^2x_h^2\}.$
        Note that even thought there are seventeen pairs of twins, we only get fifteen monomials. This is because $x_b^2x_gx_h$ and $x_bx_g^2x_h$ arise as a product of twins in two different ways. Our formula says that each of those monomials still occurs only once in $B_{\beta,\gamma}$.
    \end{example}
    
\section{Proof of Theorem \ref{thm:main-thm}} \label{sec:proof}

    We will prove Theorem \ref{thm:main-thm} by induction on $|E|$. We will consider an arbitrary monomial $\xx^\lambda$ and show that the coefficient of $\xx^\lambda$ on both sides of (\ref{eqn:main-eqn}) are equal.

    \begin{lemma}
        For each monomial appearing in (\ref{eqn:main-eqn}), the degree of each variable $x_g$ is at most two.
    \end{lemma}

    \begin{proof}
        The statement holds for the LHS because each variable in a monomial in $\TTT_e^f\TTT_f^e - \TTT_{ef}\TTT^{ef}$ has degree at most two. On the RHS, since $\beta$ and $\gamma$ are disjoint, and both are disjoint from any $\alpha\in A_{\beta,\gamma}$, any variable in $\xx^\beta\xx^\gamma$ has degree one. The monomials in $\sum_{\mm\in B_{\beta,\gamma}}\mm$ are products $\xx^\alpha\xx^{\alpha'}$, so each variable in $\mm$ has degree at most two.
    \end{proof}

    Hence, it suffices to consider monomials $\xx^\lambda$ with $\lambda_g \in \{0,1,2\}$ for each $g$. Let us say that {\it {equality holds for}} $G$ if the coefficients of this specific $\xx^\lambda$ on the left and on the right of (\ref{eqn:main-eqn}) are equal. 

    \begin{lemma}
        If $\lambda_g = 0$ for $g\in E\backslash\{e,f\}$, then if the equality holds for $G\backslash g$ then it holds for $G$. If $\lambda_g = 2$, then if the equality holds for $G/ g$ (and necessarily new $\lambda$ which differs by setting $\lambda_g = 0$) then it holds for $G$ and the old $\lambda$. Here $G\backslash g$ denotes deletion and $G/g$ denotes contraction.
    \end{lemma}

    \begin{proof}
     Assume $\lambda_g = 0$. We claim that passing from $G$ to $G\backslash g$ does not change the coefficient of $\xx^\lambda$ both on the left and on the right. On the left, this is true because the pairs of subgraphs that contribute $\xx^\lambda$ in $\TTT_e^f\TTT_f^e$ and $\TTT_{ef}\TTT^{ef}$ do not depend on whether the edge $g$ is included in $G$. The same is true for the accompanying power of $q$. On the right, this is true because whether $\alpha$ is compatible with $\beta, \gamma$, and whether $\alpha, \alpha'$ are twins do not depend on whether the edge $g$ is included in $G$. 

     Assume now $\lambda_g = 2$. If $g$ is parallel to one of $e,f$ then the coefficients of $\xx^\lambda$ on both sides are $0$. Indeed, since $\alpha, \beta, \gamma$ are all disjoint, the only way for $\lambda_g = 2$ to hold is if $g \in \alpha, \alpha'$. Then $\gamma + \alpha$ cannot be a paracel, as either $e$ or $f$ would not be connecting two distinct connected components. For the left hand side, assume $g$ is parallel to $e$. As we know pairs of graphs that contribute an instance of $\xx^\lambda$ in $\TTT_e^f\TTT_f^e$ are in bijection with pairs of graphs that contribute an instance of $\xx^\lambda$ in $\TTT_{ef}\TTT^{ef}$, via moving $e$ from one subgraph to another - see Remark \ref{remark:AB}. Due to the fact that $g$ is parallel to $e$ and $\lambda_g = 2$, moving $e$ in this way does not change the number of connected components in either of the subgraphs involved, and thus the power of $q$ is also the same, leading to cancellation. 

     If we assume now that $g$ is not parallel to either $e$ or $f$, we claim that passing from $G$ to $G/g$ does not change the coefficient of $\xx^\lambda$ both on the left and on the right. Indeed, contracting $g$ present in both subgraph does not change powers of $q$ that appear on the left hand side. Also, as $g \in \alpha, \alpha'$ as explained above, contracting it leaves $\alpha, \alpha'$ compatible with $\beta, \gamma$, and also leaves $\alpha, \alpha'$ to be twins. This completes the proof.
     \end{proof}

    Thus, we only need to consider $\xx^\lambda = \xx^{E^{ef}}$ where $\lambda_g = 1$ for all $g\in E^{ef}$. In other words, it is sufficient to prove equality of coefficients when each edge in $E^{ef}$ appears in $\lambda$ exactly once. 

\subsection{The LHS}

    First, we study the LHS of equation (\ref{eqn:main-eqn}). Recall that each pair of subsets $(A,B)$ of $E^{ef}$ contributes $(q^{k_1(A,B)} - q^{k_2(A,B)})\xx^{A+B}/(1-q)$ to $\MMM_{ef}(q)$, where
    \[ k_1(A,B) = k(A+e)+k(B+f), \quad\text{and}\quad k_2(A,B) = k(A+e+f) + k(B).\]
    Furthermore, $k(A+e) - k(A+e+f) \in \{0,1\}$, and $k(B) - k(B+f)\in\{0,1\}$, so $k_1(A,B) - k_2(A,B)\in \{-1,0,1\}$.

    If $k_1(A,B) - k_2(A,B) = 0$, then this pair contributes $0$ to $\MMM_{ef}(q)$ and hence $0$ to the LHS. If $k_1(A,B) - k_2(A,B) = -1$, then this pair contributes $q^{k_1(A,B)}\xx^{A+B}$ to $\MMM_{ef}(q)$ and hence $\xx^{A+B}$ to $\MMM_{ef}(1)$. We call such pair \textit{positive}. If $k_1(A,B) - k_2(A,B) = 1$, then this pair contributes $-q^{k_2 (A,B)}\xx^{A+B}$ to $\MMM_{ef}(q)$ and hence $-\xx^{A+B}$ to $\MMM_{ef}(1)$. We call such pair \textit{negative}.

    For $k_1(A,B) - k_2(A,B) = -1$, we need
    \[ k(A+e) - k(A+e+f) = 0, \quad\text{and}\quad k(B+f) - k(B) = -1. \]
    This means that the two vertices incident to $f$ are in the same connected components in $A+e$ but are not in the same connected components in $B$. Figure \ref{fig:PosA} shows all possible configurations regarding the connected components containing the four (not necessarily distinct) vertices incident to $e$ and $f$ in $A$. The blue loops denote the connected components, and $e$ and $f$ are the two dashed edges. Figure \ref{fig:PosB} shows the possible configurations for $B$. Note that the configurations in Figure \ref{subfig:PosAe} and \ref{subfig:PosBe} are paracels.

    \begin{figure}[h!]
     \centering
        \begin{subfigure}[b]{0.18\textwidth}
            \centering
            \includegraphics[scale = 0.8]{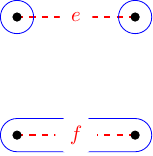}
            \caption{}
            \label{subfig:PosAa}
        \end{subfigure}
        \begin{subfigure}[b]{0.18\textwidth}
            \centering
            \includegraphics[scale = 0.8]{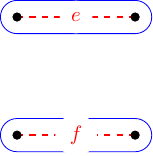}
            \caption{}
            \label{subfig:PosAb}
        \end{subfigure}
        \begin{subfigure}[b]{0.18\textwidth}
            \centering
            \includegraphics[scale = 0.8]{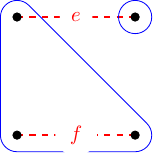}
            \caption{}
            \label{subfig:PosAc}
        \end{subfigure}
        \begin{subfigure}[b]{0.18\textwidth}
            \centering
            \includegraphics[scale = 0.8]{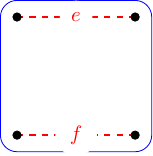}
            \caption{}
            \label{subfig:PosAd}
        \end{subfigure}
        \begin{subfigure}[b]{0.18\textwidth}
            \centering
            \includegraphics[scale = 0.8]{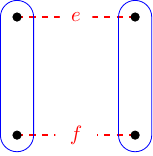}
            \caption{}
            \label{subfig:PosAe}
        \end{subfigure}

        \caption{$k(A+e) - k(A+e+f) = 0$}
        \label{fig:PosA}
    \end{figure}

    \begin{figure}[h!]
     \centering
        \begin{subfigure}[b]{0.18\textwidth}
            \centering
            \includegraphics[scale = 0.8]{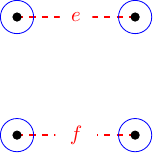}
            \caption{}
            \label{subfig:PosBa}
        \end{subfigure}
        \begin{subfigure}[b]{0.18\textwidth}
            \centering
            \includegraphics[scale = 0.8]{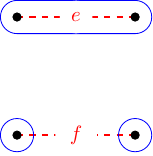}
            \caption{}
            \label{subfig:PosBb}
        \end{subfigure}
        \begin{subfigure}[b]{0.18\textwidth}
            \centering
            \includegraphics[scale = 0.8]{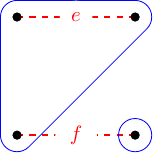}
            \caption{}
            \label{subfig:PosBc}
        \end{subfigure}
        \begin{subfigure}[b]{0.18\textwidth}
            \centering
            \includegraphics[scale = 0.8]{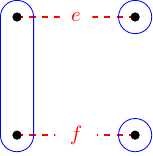}
            \caption{}
            \label{subfig:PosBd}
        \end{subfigure}
        \begin{subfigure}[b]{0.18\textwidth}
            \centering
            \includegraphics[scale = 0.8]{Fig/Type/PosAe.pdf}
            \caption{}
            \label{subfig:PosBe}
        \end{subfigure}

        \caption{$k(B+f) - k(B) = -1$}
        \label{fig:PosB}
    \end{figure}

    Similarly, for $k_1(A,B) - k_2(A,B) = 1$, we need
    \[ k(A+e) - k(A+e+f) = 1, \quad\text{and}\quad k(B+f) - k(B) = 0. \]
    This means that the two vertices incident to $f$ are not in the same connected components in $A+e$ but are in the same connected components in $B$. Figure \ref{fig:NegA} and \ref{fig:NegB} show the possible configurations for this case.

    \begin{figure}[h!]
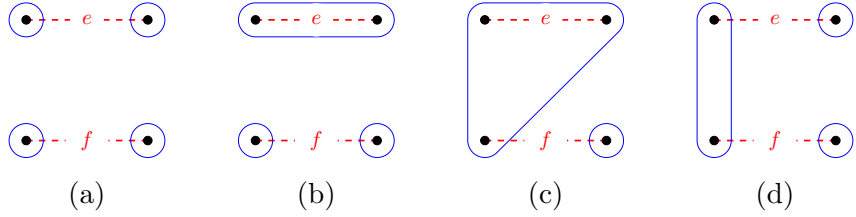

     \centering
        \begin{subfigure}[b]{0.18\textwidth}
            \centering
            \includegraphics[scale = 0.8]{Fig/Type/PosBa.pdf}
            \caption{}
            \label{subfig:NegAa}
        \end{subfigure}
        \begin{subfigure}[b]{0.18\textwidth}
            \centering
            \includegraphics[scale = 0.8]{Fig/Type/PosBb.pdf}
            \caption{}
            \label{subfig:NegAb}
        \end{subfigure}
        \begin{subfigure}[b]{0.18\textwidth}
            \centering
            \includegraphics[scale = 0.8]{Fig/Type/PosBc.pdf}
            \caption{}
            \label{subfig:NegAc}
        \end{subfigure}
        \begin{subfigure}[b]{0.18\textwidth}
            \centering
            \includegraphics[scale = 0.8]{Fig/Type/PosBd.pdf}
            \caption{}
            \label{subfig:NegAd}
        \end{subfigure}

        \caption{$k(A+e) - k(A+e+f) = 1$}
        \label{fig:NegA}
    \end{figure}

    \begin{figure}[h!]
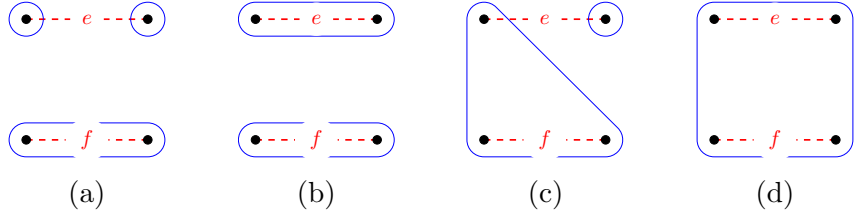

     \centering
        \begin{subfigure}[b]{0.18\textwidth}
            \centering
            \includegraphics[scale = 0.8]{Fig/Type/PosAa.pdf}
            \caption{}
            \label{subfig:NegBa}
        \end{subfigure}
        \begin{subfigure}[b]{0.18\textwidth}
            \centering
            \includegraphics[scale = 0.8]{Fig/Type/PosAb.pdf}
            \caption{}
            \label{subfig:NegBb}
        \end{subfigure}
        \begin{subfigure}[b]{0.18\textwidth}
            \centering
            \includegraphics[scale = 0.8]{Fig/Type/PosAc.pdf}
            \caption{}
            \label{subfig:NegBc}
        \end{subfigure}
        \begin{subfigure}[b]{0.18\textwidth}
            \centering
            \includegraphics[scale = 0.8]{Fig/Type/PosAd.pdf}
            \caption{}
            \label{subfig:NegBd}
        \end{subfigure}

        \caption{$k(B+f) - k(B) = 0$}
        \label{fig:NegB}
    \end{figure}

    \begin{prop}\label{prop:LHS-terms}
        The coefficient of $\xx^{G^{ef}}$ in $\MMM_ef(1)$ equals the number of paracels in $G$.
    \end{prop}

    \begin{proof}
        We only need to consider pairs of subsets $(A,E^{ef} - A)$. From Figure \ref{fig:PosA}, \ref{fig:PosB}, \ref{fig:NegA}, \ref{fig:NegB}, one can see a clear correspondence between negative pairs $(A,E^{ef}-A)$ and positive pairs $(A',E^{ef}-A')$ such that neither $A'$ nor $E^{ef}-A'$ is a paracel. This correspondence is given by $(A,E^{ef}-A) \leftrightarrow(E^{ef}-A,A)$.

        Hence, the only remaining pairs are $(A,E^{ef}-A)$ such that at least one of $A$ and $E^{ef}-A$ is a paracel, and these pairs each contributes $+1$ to the coefficient of $\xx^{G^{ef}}$.
    \end{proof}

\subsection{The RHS}

    We now show a similar statement as Proposition \ref{prop:LHS-terms} for the RHS. Consider a term $\xx^\beta\xx^\gamma\mm$ on the RHS. Note that $\mm$ is square-free, so it is a product $\xx^\alpha\xx^{\alpha'}$ for some $\alpha,\alpha'$ compatible with $\beta,\gamma$ and $\alpha\cap\alpha' = \varnothing$. This means that $\gamma+\alpha\cap\alpha' = \gamma$ is itself a paracel. Thus, every term $\xx^\beta\xx^\gamma\mm = \xx^{G^{ef}}$ on the RHS has to come from a paracel $\gamma$. We will show that the converse is true.

    \begin{prop}\label{prop:RHS-term}
        For each paracel $\gamma$, there is exactly one choice of $\beta$ and $\mm$ such that $\xx^\beta\xx^\gamma\mm = \xx^{G^{ef}}$.
    \end{prop}

    \begin{proof}
        First, we show that there is exactly one choice of $\beta$. Since $\varnothing$ is compatible with $\beta,\gamma$, every edge in $\beta$ has to be a smoot for $\gamma$. Conversely, every edge that is a smoot for $\gamma$ has to be in $\beta$. Indeed, if a smoot $g$ is not in $\beta$, then $x_g$ is in $\mm$. This means that $g$ is in some $\alpha$ compatible with $\beta,\gamma$. This contradicts the requirement that $\gamma+\alpha$ is a paracel.

        Finally, we need to show that for each paracel $\gamma$ and $\beta = \{\text{smoots in $\gamma$}\}$, there is always at least one way to split the remaining edges into disjoint $\alpha$ and $\alpha'$ compatible with $\beta,\gamma$. Let $\Sigma$ be the set of the remaining edges, $C_1$ and $C_2$ be the two connected components connected by $e$ and $f$ in $\eta(\gamma)$. Let $\alpha$ be a (possibly empty) subset of $\Sigma$ consisting of all edges that have at least one endpoint in $C_1$. Let $\alpha' = \Sigma - \alpha$, clearly $\alpha$ and $\alpha'$ are disjoint.
        
        We claim that both $\alpha$ and $\alpha'$ are compatible with $\beta,\gamma$. Every edge in $\alpha'$ connects two vertices not in $C_1$. Hence, $C_1$ and $C_2$ are still disconnected in $\eta(\gamma + \alpha')$, so $\gamma + \alpha'$ is a paracel. We are left to show that $\gamma+\alpha$ is also a paracel. Suppose otherwise that $C_1$ and $C_2$ are connected in $\eta(\gamma + \alpha)$. This means that in $\eta(\gamma + \alpha)$, there is a path from a vertex in $C_1$ to a vertex in $C_2$ with edges in $\alpha$. However, every edge in $\alpha$ has an endpoint in $C_1$, so this means there is an edge between a vertex in $C_1$ and a vertex in $C_2$. This edge is a smoot, which is a contradiction since every smoot is in $\beta$.
    \end{proof}

    \begin{example}\label{ex:RHSEx1}
        In Figure \ref{fig:RHSEx1}, $e$ and $f$ are the red dashed edges. Let $\gamma$ contain the blue edges, then all remaining (orange) edges are smoots, so they are all in $\beta$. Then, $\alpha = \alpha' = \varnothing$.
    
        \begin{figure}[h!]
            \centering
            \includegraphics[scale = 0.6]{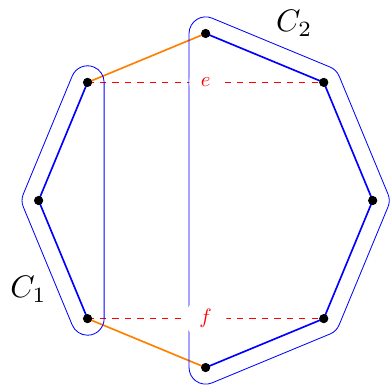}
            \caption{}
            \label{fig:RHSEx1}
        \end{figure}
    \end{example}

    \begin{example}
        In Figure \ref{fig:RHSEx2}, $e,f,\gamma$ and $\beta$ are the same as in Example \ref{ex:RHSEx1}. Among the remaining (black) edges, none has an endpoint in $C_1$. Thus, they are all in $\alpha'$ and $\alpha = \varnothing$.
    
        \begin{figure}[h!]
            \centering
            \includegraphics[scale = 0.6]{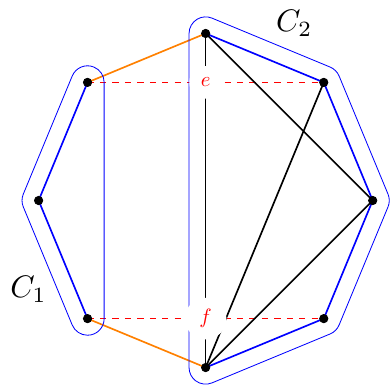}
            \caption{}
            \label{fig:RHSEx2}
        \end{figure}
    \end{example}

    \begin{example}
        In Figure \ref{fig:RHSEx3}, $e$ and $f$ are colored red, and $\gamma$ is colored blue. The orange edges are forces to be in $\beta$. Among the remaining edges, the green edges have at least one endpoint in $C_1$, so they are in $\alpha$. The remaining edges are in $\alpha'$. The readers can verify from Figure \ref{subfig:RHSEx3alpha} and \ref{subfig:RHSEx3alphap} that $\gamma + \alpha$ and $\gamma + \alpha'$ are both paracels.
    
        \begin{figure}[h!]
         \centering
            \begin{subfigure}[b]{0.3\textwidth}
                \centering
                \includegraphics[scale = 0.6]{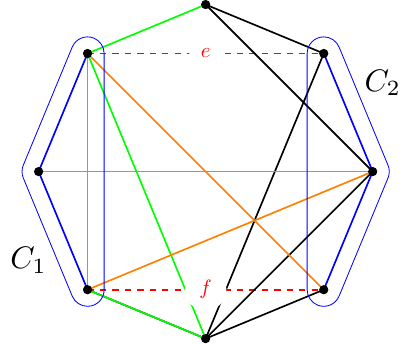}
                \caption{$G$}
                \label{subfig:RHSEx3}
            \end{subfigure}
            \begin{subfigure}[b]{0.3\textwidth}
                \centering
                \includegraphics[scale = 0.6]{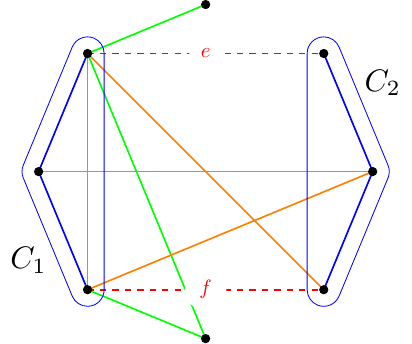}
                \caption{$\beta + \gamma + \alpha$}
                \label{subfig:RHSEx3alpha}
            \end{subfigure}
            \begin{subfigure}[b]{0.3\textwidth}
                \centering
                \includegraphics[scale = 0.6]{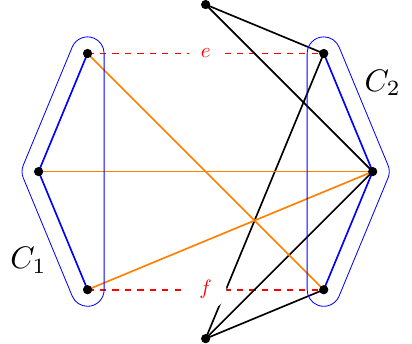}
                \caption{$\beta + \gamma + \alpha'$}
                \label{subfig:RHSEx3alphap}
            \end{subfigure}
    
            \caption{}
            \label{fig:RHSEx3}
        \end{figure}
    \end{example}

\section{$\alpha \beta \gamma$-ansatz} \label{sec:conj}

    We conjecture that $\MMM_{ef}(q)$ in general can be written in terms of $\beta,\gamma$, and $A_{\beta,\gamma}$ similar to Theorem \ref{thm:main-thm}.

    \begin{conjecture}\label{conj:general-q}
        For any graph $G$ and $0\leq q\leq 1$, the polynomial $\MMM_{ef}(q)$ can be written in the form
        \[ \dfrac{\MMM_{ef}(q)}{q^2} = \sum_{\beta,\gamma} \xx^\beta\xx^\gamma Q_{\beta,\gamma}, \]
        where $Q_{\beta,\gamma}$ is a positive semidefinite quadratic form in the variables $\{\xx_\alpha~|~\alpha\in A_{\beta,\gamma}\}$. Here, $A_{\beta,\gamma}$ is the set of $\alpha$'s compatible with $\beta$, $\gamma$ as in Definition \ref{def:compatible}.
        
        Furthermore, setting $q = 0$ and taking lowest degree terms in $Q_{\varnothing,\varnothing}$ gives the classical UST formula in Theorem \ref{thm:trees}.
    \end{conjecture}

    Here we give a few examples illustrating Conjecture \ref{conj:general-q}.

    \begin{example}
        For $G = K_3$ with three edges $e,f,g$. The pairs of $\beta$ and $\gamma$ with nonzero $Q_{\beta,\gamma}$ are
        \begin{center}
            \begin{tabular}{c | c | c | c} 
                 $\beta$ & $\gamma$ & $A_{\beta,\gamma}$ & $Q_{\beta,\gamma}$ \\
                 \hline\hline
                 $\varnothing$ & $\varnothing$ & $\{\{g\}\}$ & $x_g^2$ \\
                 \hline
                 $\varnothing$ & $\{g\}$ & $\{\varnothing\}$ & $q$
            \end{tabular}
        \end{center}

        Hence,
        \[ \dfrac{\MMM_{ef}(q)}{q^2} = \xx^{\varnothing}\cdot \xx^{\varnothing}\cdot x_g^2 + \xx^{\varnothing}\cdot \xx^{\{g\}}\cdot q = x_g^2 + qx_g, \]
        which is consistent with Example \ref{ex:k3q}. 
        
        Setting $q = 0$ and taking lowest degree terms in $Q_{\varnothing,\varnothing}$ give $x_g^2$, which is the UST formula since $\{g\}$ is the only paracel that is a spanning forest with two connected components.
    \end{example}

    \begin{example}\label{ex:qK4-1}
        Revisiting the graph in Example \ref{ex:K4-1}, we have the following terms.

        \begin{center}
            \begin{tabular}{c | c | c | c} 
                 $\beta$ & $\gamma$ & $A_{\beta,\gamma}$ & $Q_{\beta,\gamma}$ \\
                 \hline\hline
                 $\varnothing$ & $\varnothing$ & $\{\{g\},\{g,h\},\{g,k\}\}$ & $(qx_g + x_gx_h + x_gx_k)^2$ \\
                 \hline
                 $\varnothing$ & $\{g\}$ & $\{\varnothing,\{h\},\{k\}\}$ & $q(q + x_h + x_k)^2$ \\
                 \hline
                 $\varnothing$ & $\{h\}$ & $\{\{g\}\}$ & $0$ \\
                 \hline
                 $\varnothing$ & $\{k\}$ & $\{\{g\}\}$ & $0$ \\
                 \hline
                 $\varnothing$ & $\{g,h\}$ & $\{\varnothing\}$ & $0$ \\
                 \hline
                 $\varnothing$ & $\{g,k\}$ & $\{\varnothing\}$ & $0$ \\
                 \hline
                 $\{h\}$ & $\varnothing$ & $\{\{g,k\}\}$ & $x_g^2x_k^2$ \\
                 \hline
                 $\{h\}$ & $\{g\}$ & $\{\{k\}\}$ & $qx_k^2$ \\
                 \hline
                 $\{h\}$ & $\{k\}$ & $\{\{g\}\}$ & $qx_g^2$ \\
                 \hline
                 $\{h\}$ & $\{g,k\}$ & $\{\varnothing\}$ & $q^2$ \\
                 \hline
                 $\{k\}$ & $\varnothing$ & $\{\{g,h\}\}$ & $x_g^2x_h^2$ \\
                 \hline
                 $\{k\}$ & $\{g\}$ & $\{\{h\}\}$ & $qx_h^2$ \\
                 \hline
                 $\{k\}$ & $\{h\}$ & $\{\{g\}\}$ & $0$ \\
                 \hline
                 $\{k\}$ & $\{g,h\}$ & $\{\varnothing\}$ & $0$ \\
            \end{tabular}
        \end{center}

        Note that there are some pairs of $\beta$ and $\gamma$ with $Q_{\beta,\gamma} = 0$ but nonempty $B_{\beta,\gamma}$ in Example \ref{ex:K4-1}. This is because we need these monomials in $(qx_g + x_gx_h + x_gx_k)^2$ and $q(q + x_h + x_k)^2$ in order to obtain positive semidefinite quadratic forms.

        Setting $q = 0$ and taking lowest degree terms in $Q_{\varnothing,\varnothing}$ give
        \[ (x_gx_h + x_gx_k)^2, \]
        which is the UST formula.
    \end{example}

    \begin{remark}
        Example \ref{ex:qK4-1} suggests that from Theorem \ref{thm:main-thm}, we can group monomials from different pairs of $(\beta,\gamma)$ together to form positive semidefinite quadratic forms. However, it is not clear to us which pairs of $(\beta,\gamma)$ should be grouped together.
        
        For instance, in Example \ref{ex:qK4-1}, we merged the monomial $q^2x_gx_hx_k$ from $(\beta,\gamma) = (\{k\}, \{g,h\})$ into $(\beta,\gamma) = (\varnothing,\{g\})$. Notice that this monomial also appears in $(\beta,\gamma) = (\{h\}, \{g,k\})$, so we could have instead merged this pair instead. Another option is to merge $\dfrac{1}{2}q^2x_gx_hx_k$ from each pair, but this approach quickly becomes intractable as the graph grows bigger.
    \end{remark}

    \begin{example}
         Let $G$ and $e,f$ as in Figure \ref{fig:K4minusOne2}.
        
        \begin{figure}[h!]
            \centering
            \includegraphics[scale = 0.8]{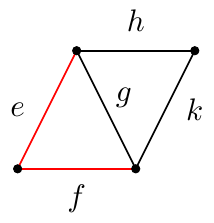}
            \caption{}
            \label{fig:K4minusOne2}
        \end{figure}

        We have the following terms.

        \begin{center}
            \begin{tabular}{c | c | c | c} 
                 $\beta$ & $\gamma$ & $A_{\beta,\gamma}$ & $Q_{\beta,\gamma}$ \\
                 \hline\hline
                 $\varnothing$ & $\varnothing$ & $\{\{g\},\{g,h\},\{g,k\},\{g,h,k\}\}$ & $(qx_g + x_gx_h + x_gx_k + x_gx_hx_k)^2$ \\
                 \hline
                 $\varnothing$ & $\{g\}$ & $\{\varnothing,\{h\},\{k\},\{h,k\}\}$ & $q(x_h+x_k+q)^2$ \\
                 \hline
                 $\varnothing$ & $\{h\}$ & $\{\{g\},\{k\},\{g,k\}\}$ & $qx_k^2$ \\
                 \hline
                 $\varnothing$ & $\{k\}$ & $\{\{g\},\{h\},\{g,h\}\}$ & $qx_h^2$ \\
                 \hline
                 $\varnothing$ & $\{g,h\}$ & $\{\varnothing,\{k\}\}$ & $qx_k^2$ \\
                 \hline
                 $\varnothing$ & $\{g,k\}$ & $\{\varnothing,\{h\}\}$ & $qx_h^2$ \\
                 \hline
                 $\varnothing$ & $\{h,k\}$ & $\{\varnothing,\{g\}\}$ & $q^2$ \\
                 \hline
                 $\varnothing$ & $\{g,h,k\}$ & $\{\varnothing\}$ & $q^2$ \\
            \end{tabular}
        \end{center}

        Notice that for $(\beta,\gamma) = (\varnothing,\{h\})$, there are three subsets in $A_{\beta,\gamma}$, but the quadratic form $Q_{\beta,\gamma}$ only has $qx_k^2$. This also happens with other pairs of $(\beta,\gamma)$.

        Setting $q = 0$ and taking lowest degree terms in $Q_{\varnothing,\varnothing}$ give
        \[ (x_gx_h + x_gx_k)^2, \]
        which is the UST formula.
    \end{example}

    \begin{example}
        Let $G = K_4$ as in Figure \ref{fig:K4} with $e,f$ colored red.
        
        \begin{figure}[h!]
            \centering
            \includegraphics[scale = 0.8]{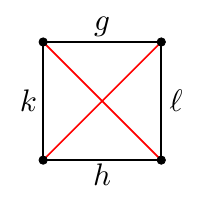}
            \caption{}
            \label{fig:K4}
        \end{figure}

        We have the following terms.

        \begin{center}
            \begin{tabular}{c | c | c | c} 
                 $\beta$ & $\gamma$ & $A_{\beta,\gamma}$ & $Q_{\beta,\gamma}$ \\
                 \hline\hline
                 $\varnothing$ & $\varnothing$ & $\{\{g,h\},\{k,\ell\}\}$ & $(x_gx_h)^2 + (x_kx_\ell)^2 - (2-3q+q^2)x_gx_hx_kx_\ell$ \\
                 \hline
                 $\varnothing$ & $\{g\}$ & $\{\{h\}\}$ & $qx_h^2$ \\
                 \hline
                 $\varnothing$ & $\{h\}$ & $\{\{g\}\}$ & $qx_g^2$ \\
                 \hline
                 $\varnothing$ & $\{k\}$ & $\{\{\ell\}\}$ & $qx_\ell^2$ \\
                 \hline
                 $\varnothing$ & $\{\ell\}$ & $\{\{k\}\}$ & $qx_k^2$ \\
                 \hline
                 $\varnothing$ & $\{g,h\}$ & $\{\varnothing\}$ & $q^2$ \\
                 \hline
                 $\varnothing$ & $\{k,\ell\}$ & $\{\varnothing\}$ & $q^2$ \\
                 \hline
                 \multirow{4}{6.5em}{$\{g\},\{h\},\{g,h\}$} & $\varnothing$ & $\{\{k,\ell\}\}$ & $x_k^2x_\ell^2$ \\
                 & $\{k\}$ & $\{\{\ell\}\}$ & $qx_\ell^2$ \\
                 & $\{\ell\}$ & $\{\{k\}\}$ & $qx_k^2$ \\
                 & $\{k,\ell\}$ & $\{\varnothing\}$ & $q^2$ \\
                 \hline
                 \multirow{4}{6.5em}{$\{k\},\{\ell\},\{k,\ell\}$} & $\varnothing$ & $\{\{g,h\}\}$ & $x_g^2x_h^2$ \\
                 & $\{g\}$ & $\{\{h\}\}$ & $qx_h^2$ \\
                 & $\{h\}$ & $\{\{g\}\}$ & $qx_g^2$ \\
                 & $\{g,h\}$ & $\{\varnothing\}$ & $q^2$ \\
            \end{tabular}
        \end{center}

        Note that $Q_{\varnothing,\varnothing}$ is no longer a perfect square, but it is still positive semidefinite for $0\leq q \leq 1$. Nevertheless, setting $q = 0$ and taking lowest degree terms in $Q_{\varnothing,\varnothing}$ does give a perfect square
        \[ (x_gx_h - x_kx_\ell)^2, \]
        which is the UST formula.
    \end{example}

\bibliography{bibliography}
\bibliographystyle{alpha}

\end{document}